\documentclass[12pt,reqno]{amsart}
\usepackage{amsmath,amssymb,latexsym,textcomp,mathrsfs}
\usepackage[all]{xy}
\usepackage{graphicx}
\usepackage{bm,amsmath, amsthm, amssymb, amsfonts}
\usepackage{times}
\usepackage[english]{babel}

\setbox0=\hbox{$+$}
\newdimen\plusheight
\plusheight=\ht0
\def\+{\;\lower\plusheight\hbox{$+$}\;}

\setbox0=\hbox{$-$}
\newdimen\minusheight
\minusheight=\ht0
\def\-{\;\lower\minusheight\hbox{$-$}\;}

\setbox0=\hbox{$\cdots$}
\newdimen\cdotsheight
\cdotsheight=\plusheight
\def\cds{\lower\cdotsheight\hbox{$\cdots$}}

\numberwithin{equation}{section}
\theoremstyle{plain}
\newtheorem{theorem}{Theorem}[section]
\newtheorem{lemma}{Lemma}[section]

\newtheorem{corollary}{Corollary}[section]
\newtheorem{definition}{Definition}[section]

\newtheorem{example}{Example}[section]

\newtheorem{remark}{Remark}[section]
\newtheorem{note}{Note}[section]

  \newenvironment{nouppercase}{%
   \renewcommand{\uppercasenonmath}[1]{}}{}
	
	 \newcommand{\Keywords}[1]{\par\noindent
   {\small{Keywords}: #1}}
   
   \newcommand{\AMS}[1]{\par\noindent
   {\small{AMS Subject Classification (2010)}: #1}}

\begin{document}

\title{$ Ig^*$-CLOSED SETS AND ITS CERTAIN PROPERTIES IN $ \sigma $-SPACES WITH RESPECT TO AN IDEAL}

 \author{Jagannath Pal$^1$}
 \author{Amar Kumar Banerjee$^2$}
 \newcommand{\acr}{\newline\indent}
 \maketitle
 \address{{1\,} Department of Mathematics, The University of Burdwan, Golapbag, East Burdwan-713104,
 West Bengal, India.
 Email: jpalbu1950@gmail.com\acr
 {2\,} Department of Mathematics, The University of Burdwan, Golapbag, East Burdwan-713104,
 West Bengal, India. Email:akbanerjee@math.buruniv.ac.in, akbanerjee1971@gmail.com 
 \\}
   
\begin{abstract}
Here we have introduced and studied the idea of $ Ig^*$-closed set with respect to an ideal and investigated some of its properties in  Alexandroff spaces \cite{AD}. We have also introduced $ Ig^*$-$T_0 $ axiom,  $ Ig^*$-$T_1$ axiom, $ Ig^*$-$T_\omega $ axiom and explored a relation among them.
\end{abstract}

\begin{nouppercase}
\maketitle
\end{nouppercase}

\let\thefootnote\relax\footnotetext{
\AMS{Primary 54A05, 54D10}
\Keywords {$ Ig^* $-closed sets, $ Ig^*$-$T_0 $ axiom,  $ Ig^*$-$T_1$ axiom, $ Ig^*$-$T_\omega $ axiom.}

}

\section{\bf Introduction}
\label{sec:int}
In (1940), A. D. Alexandroff \cite{AD} generalized topological space to a $ \sigma $-space (Alexandroff space) by changing only arbitrary union condition where countable union of open sets were taken to be open. Simultaneously, several authors introduced generalized closed sets and they contributed lot of works on generalized closed sets on several spaces \cite{AJ1, AC, CA, DP, WD, SM, MS} etc. to develop the modification. The idea of generalized closed sets in a topological space was given by Levine \cite{NL}. In 2003, P.  Das et al obtained a generalization of closed sets in Alexandroff spaces which was called $ g^* $-closed sets. They investigated various properties of it and introduced a new separation axiom namely $ T_\omega $ axiom in a $ \sigma $-space in the same way as that of $ T_\frac{1}{2} $-spaces introduced by Levine \cite{NL}  in a topological space. Introduction of ideals in topological spaces added new dimension to the subject and many topologists have been working with it for a long time on different areas of general topology \cite{AA, AAB, AP, IV, JH,JR, BP, DJ, SV} on different spaces.    

In this paper, we have studied the notion of generalized closed sets in Alexandroff spaces with respect to an ideal by introducing $ Ig^*$-closed set and  investigated the effects of topological properties with the newly introduced generalized closed set. Also we have obtained $ Ig^*$-$T_0 $ axiom,  $ Ig^*$-$T_1$ axiom, $ Ig^*$-$T_\omega $ axiom in an Alexandroff space and tried to establish a relation among them. Examples are offered where necessary to enrich the paper.

 \section{\bf Preliminaries}
 \label{sec:pre}

\begin{definition}\label{1}
\cite{AD}: An Alexandroff space (or $ \sigma $-space, briefly space) is a set $ Y $ together with a system  $\mathcal{F}$ of subsets satisfying the following axioms

(1)	the intersection of a countable number of sets from $\mathcal{F}$ is a set in $\mathcal{F}$

(2) the union of a finite number of sets from $\mathcal{F}$ is a set in $\mathcal{F}$

(3)	$ \emptyset $ and $ Y $ are in $ \mathcal{F} $

Members of $ \mathcal{F}$ are called closed sets and their complements are open sets. Clearly  we can take open sets in lieu of closed sets in the definition of $ \sigma $-space subject to the conditions of countable summability, finite intersectability, and the condition that $ Y $ and $ \emptyset $ should be open. The collection of all such open sets is denoted by $\eta $ and the space by    $(Y,\eta)$. When there is no confusion, the space  $(Y , \eta)$ will simply be denoted by $ Y $.
\end{definition}

It can be easily verified that a topological space is a space but in general $ \eta $ is not a topology. Several examples of spaces are seen in \cite{PD}, \cite{DP}, \cite{LD}. Definitions of closure, interior of a set and $ T_0,  T_1 $ axioms in a space are taken parallel to those of a topological space. Obviously, closure of a set $ D $ in a space, denoted by $ \overline{D} $, may not be closed in general.

\begin{definition} \cite{BP}:  If $ Y $ is a nonvoid set then a family of sets $ I\subset 2^X $ is an ideal if

(i) $ \emptyset \in I $, 

(ii) $ C, D\in I $ implies $ C\cup D\in I $ and 

(iii) $ C\in I, D\subset C $ implies $  D\in I $.
\end{definition} 

If $ (Y,\mu) $ is a topological space then the triplet $ (Y,\mu, I) $ is called an ideal topological space.

\begin{definition}\cite{JH}\label{2}:
If $ (Y,\eta) $ is a $ \sigma $-space and $ I $ is an ideal on $ Y $, the triplet $ (Y,\eta, I ) $ is called ideal space (or, ideal $ \sigma $-space).  
\end{definition}

\begin{remark}
In many papers where the idea of ideal convergence were studied elaborately, it was conventionally taken that the ideal is trivial if (i) $ I=\{\emptyset\} $ or, (ii) if the whole set belongs to $ I $. But in our present discussion we take only the first condition.
\end{remark}

\begin{definition}\label{3}
\cite{WD} :  Two sets $ C, D $ of $ Y $ are said to be weakly separated if there are two open sets $ E, F $ such that $C\subset E, D\subset F $ and $ C\cap F=D\cap E=\emptyset $.
\end{definition}

\begin{definition}\label{4}
\cite{NL}: A set $C$ in a topological space is said to be generalized closed ($g$-closed for short) if  $\overline {C} \subset E$ whenever $C\subset E$ and $ E $ is open.
\end{definition}

\begin{definition}\label{5}
\cite {DP}:  A set $ C $ is said to be  $g^*$-closed  in a space $ (Y, \eta) $ if there is a closed set  $ P $ containing $ C $ such that $ P \subset E $ whenever $C\subset E $ and  $  E $ is open.  $A$ is called $ g^*$-open if $ Y - C $ is $ g^*$-closed.
\end{definition}

\begin{definition}\cite{MS}\label{6}:
Let $ D $ be a subset of a space $ (Y, \eta) $. We define kernel of $ D $ denoted by $ ker(D)=\bigcap \{U: U\supset D; U $ is open\}. 
\end{definition}

\begin{lemma}\label{7}
Let $ C\subset Y $, then $ Int(C)=Y-\overline{Y-C} $. 
\end{lemma}

Clearly every closed set is $g^*$-closed and every open set is $ g^* $-open but converse may not be true in general as it can be verified from Example 1 \cite{DP}. Also union of two $ g^* $-closed sets is $ g^* $-closed  \cite{DP} and so intersection of two $ g^* $-open sets is $ g^* $-open.

\section{\bf    $\textit Ig^
*$-closed sets in an ideal $ \sigma $-space}

In this section we study the idea of $ Ig^* $-closed set, a generalization of closed sets with respect to an ideal and investigate some of its topological properties in the setting of a $ \sigma $-space.

\begin{definition}(c.f.\cite{JR})\label{8}:
A subset $ C $ in an ideal $ \sigma $space $ (Y, \eta,I) $ is said to be an $ Ig^* $-closed if there is a closed set  $ P $ containing $ C $ such that $ P-E\in I $ whenever $ C\subset E $ and  $  E $ is open.  $ C $ is called an $ Ig^*$-open set if $ Y - C $ is $ Ig^*$-closed. 
\end{definition}

\begin{remark}\label{9}
Obviously in a $ \sigma $-space, a closed set is $ Ig^*$-closed and a $ g^* $-closed set is also $ Ig^*$-closed  but converse may not be always true as shown by Examples \ref{10}. Analogously, an open set is $ Ig^*$-open and a $ g^* $-open set is also $ Ig^*$-open.
\end{remark}

\begin{example}\label{10} (i):
Example of an $ Ig^* $-closed set which is not closed in an ideal $ \sigma $-space.

Let $ Y=R-Q, \eta=\{\emptyset,Y,D_i\} $ where $ D_i $ are all countable subsets of $ Y $ and $ I=\{\emptyset, \{\sqrt{7}\}\} $. Then $ (Y,\eta,I) $ is an ideal $ \sigma $-space. Assume $ D $ is the set of all irrational numbers of $ (1,2) $, then $ D $ is not closed but $ Ig^* $-closed since $ Y $ is the only open set containing $ D $. 

(ii): Example of an $ Ig^* $-closed set which is not $ g^* $-closed in an ideal $ \sigma $-space.

Let $ Y=R-Q, \eta=\{\emptyset, Y, \{\sqrt{3}\}\cup D_i\} $ where $ D_i $ are all countable subsets of $ Y $ and $ I =\{\emptyset $, all subsets of $ Y $ which do not contain $ \sqrt{3} $ \}. Then $ (Y,\eta,I) $ is an ideal $ \sigma $-space. Assume $ D=\{\sqrt{3}, \sqrt{7}\} $. Since $ Y $ is the only closed set containing $ D $, then $ D $ is $ Ig^* $-closed as $ Y-U\in I $ for any open set $ U, U\supset D $ but $ D $ is not a $ g^* $-closed set by Definition \ref{5}.  
\end{example}

\begin{theorem}\label{11}
Let $ I $ be a trivial ideal, then a set $ D\subset (Y,\eta,I) $ is $ Ig^* $-closed if and only if there is a closed set $ P, P\supset D $ such that $ P-ker(D)\in I $.
\end{theorem}

\begin{proof}
Let $ I $ be a trivial ideal and $ D $ be an  $ Ig^* $-closed set of $ (Y,\eta,I) $. Then there is a closed set  $ P $ containing $ D $ such that $ P-E\in I=\{\emptyset\} $ whenever $ D\subset E $ and  $  E $ is any open set which implies that $ P-ker(D)\in I=\{\emptyset\} $.

Conversely, suppose the conditions hold and take any open set $ E $ containing $ D $. Then $ P-E\in I $. This implies that $ D $ is an  $ Ig^* $-closed set.
\end{proof}

\begin{corollary}\label{12}
Let $ I $ be a trivial ideal; then a set $ D $ of $ (Y,\eta,I) $ satisfying $ D=ker(D) $ is $ Ig^* $-closed if and only if $ D $ is closed.
  
Proof is simple and so is omitted.
\end{corollary}

\begin{theorem}\label{13}
An open $ Ig^*$-closed set $ C\subset (Y,\eta, I) $ is closed  if the ideal is trivial.
\end{theorem} 
\begin{proof}
Let $ C $ be $ Ig^*$-closed and open. Then by Definition \ref{8}, there is a closed set $ P $ containing $ C $ such that $ P-C\in I=\{\emptyset\} $ where $ C $ is open $ \Rightarrow P=C \Rightarrow C $ is closed. 
\end{proof}

\begin{theorem}\label{14}
A set $ C\subset (Y,\eta, I) $ is $ Ig^*$-closed  if and only if there is a closed set $ P $ containing $ C $ such that $ F\subset P - C $ for some closed set $ F $ implies $ F\in I $. \end{theorem}

\begin{proof}
Let $ C $ be an $ Ig^*$-closed set of $ (Y,\eta, I) $. Then there is a closed set $ P $ containing $ C $ such that $ P-U\in I $ whenever $ U $ is open and $ U\supset C $. Assume $ F\subset P -C $ and $ F $ is closed. So $ Y-F $ is open and $ C\subset Y-F $. Hence by definition, $ P-(Y-F)\in I\Rightarrow F=P\cap F=P-(Y-F)\in I $.

Conversely, suppose the conditions hold and $ U $ is an open set containing $ C $, then $ P-U\subset P-C $. Let $ F=P-U=P\cap (Y-U) $. Then $ F $ is a closed set and $ F\subset P-C $ and so by assumption, $ F\in I $ and hence $ C $ is $ Ig^*$-closed.
\end{proof}

\begin{corollary}\label{15}
Suppose $ C $ is an $ Ig^*$-closed set of $ (Y,\eta,I) $,  $ \overline{C} $ and $ \overline{C}-C $ are both closed, then $ C $ is closed if $ I $ is a trivial ideal.
\end{corollary}

\begin{proof}
Suppose the conditions hold. Since $ C $ is $ Ig^*$-closed then by Theorem \ref{14}, there is a closed set $ P $ containing $ C $ such that $ F\subset P-C $ for some closed set $ F \Rightarrow F\in I $. Now $ C\subset P\Rightarrow  \overline{C}\subset \overline{P}=P\Rightarrow\overline{C}-C\subset P-C $, $ \overline{C}-C $ is closed by assumption, then $ \overline{C}-C\in I=\emptyset$. So $ \overline{C}=C $ and hence $ C $ is a closed set.
\end{proof}

\begin{theorem}\label{16}
:   For each $y \in  (Y,\eta,I) ,$  $\{y\}$ is either closed or $ Y - \{y\}$ is $ Ig^*$-closed.
\end{theorem}
\begin{proof}
Suppose $ \{y\} $ is not closed. Then $ Y-\{y\} $ is not open. So $ Y $ is the only open set containing $ Y-\{y\}$. Thus we see that there is a closed set $ Y $ containing $ Y-\{y\}$ such that  $ Y-Y=\emptyset\in I $ whenever $ Y $ is an open set. Hence, $ Y-\{y\} $ is $ Ig^*$-closed.
\end{proof}

\begin{theorem}\label{17}
Let $ C_1 $ and $ C_2 $ be two $ Ig^* $-closed sets in $ (Y,\eta, I) $ then $ C_1\cup  C_2 $ is $ Ig^* $-closed.
\end{theorem}
\begin{proof}
 Suppose $ C_1 $ and $ C_2 $ are two $ Ig^* $-closed sets in an ideal $ \sigma $-space $ (Y,\eta, I) $ and $ U $ is an open set such that $ C_1\cup C_2\subset U $. Then $ C_1\subset U $ and $ C_2\subset U $. Therefore, there exist closed sets $ P_1 $ and $ P_2 $ containing  $ C_1 $ and $ C_2 $ respectively such that $ P_1-U\in 
 I $ and $ P_2-U\in 
 I $. Then $ (P_1\cup P_2)\supset  (C_1\cup C_2) $ and $ (P_1\cup P_2) $ is a closed set such that $ (P_1\cup P_2)-U\in I $. Hence the result follows.
 \end{proof}

As in the case of a topological space, intersection of two  $ Ig^* $-closed sets may not be $ Ig^* $-closed as seen from the Example \ref{18}

\begin{example}\label{18} Example of intersection of two  $ Ig^* $-closed sets may not be $ Ig^* $-closed.

Let $ Y=R-Q, \eta=\{\emptyset,Y,\{\sqrt{11}\}\cup V_i\} $ where $ V_i $ are all countable subsets of $ Y $ and $ I=\{\emptyset, \{\sqrt{17}\}\} $. Then $ (Y,\eta,I) $ is an ideal $ \sigma $-space but not an ideal topological space. Let $ C $  and $ D $ be the sets of all irrational numbers in $ (1,2) $ and $ (2,3) $ together with $ \sqrt{11} $ each respectively. So $ Y $ is the only open and closed sets containing $ C $ or $ D $; hence $ C $ and $ D $ are $ Ig^* $-closed sets. Now $ C\cap D=\{\sqrt{11}\} $ and $ Y $ is the only closed set containing $ C\cap D $, so $ C\cap D $ is not an $ Ig^* $-closed set since $ Y- $ open set containing $ \{\sqrt{11}\}   \not \in I $.  
\end{example}

\begin{theorem}\label{19}
If $ D\subset (Y,\eta,I) $ is $ Ig^*$-closed and $ D\subset C\subset \overline{D} $ then $ C $ is $ Ig^*$-closed.
\end{theorem}

\begin{proof}
Suppose $ D $ is an $ Ig^*$-closed set in $ (Y,\eta, I),   D\subset C\subset \overline{D} $ and $ C\subset U, U $ is open. Then $ D\subset U $ and hence  there is a closed set $ P, P\supset D $ such that $ P-U\in I $. Now $ D\subset P\Rightarrow \overline{D}\subset P \Rightarrow P\supset \overline{D}\supset C $ and $ P-U\in I $. Hence, $ C $ is  $ Ig^*$-closed.
\end{proof}

\begin{theorem}\label{20}
Suppose $ D\subset C\subset (Y,\eta,I) $ when $ C $ is open and $ g^* $-closed, then $ D $ is $ Ig^* $-closed relative to $ C $ with respect to the ideal $ I_C=\{O\subset C: O\in I\} $ if and only if $ D $ is $ Ig^* $-closed.
\end{theorem}

\begin{proof}
Let $ (Y, \eta,I) $ be an ideal $ \sigma $-space and $ D\subset C $ where $ C $ is open and $ g^* $-closed. Then by Definition \ref{5}, there exists a closed set $ M $ containing $ C $ such that $ M\subset C $ (since $ C\subset C, C $ is open) and hence $ M=C\Rightarrow C $ is closed.

Suppose $ D $ is $ Ig^* $-closed relative to $ C $ with respect to the ideal $ I_C $. Then there exists a closed set $ P $ in $ C, P\supset D $,  such that $ P-U\in I_C $ where $ U $ is open in $ C $. Now $ P=C\cap K $ where $ K $ is closed in $ Y $ and so $ P $ is closed in $ Y $. Let $ D\subset U_1 $ be an open set in $ Y $, then $ D\subset U_1\cap C $, an open set in $ C $. Since $ D $ is $ Ig^* $-closed relative to $ C $ with respect to the ideal $ I_C $, $ P-(U_1\cap C)=V $ (say), for some $ V\in I_C $ and since $ U_1\cap C\subset U_1 $ so $ P-U_1\subset V \Rightarrow P-U_1\in I $ and  hence $ D $ is $ Ig^* $-closed in $ Y $.

Let $ D $ be $ Ig^*$-closed and $ D\subset C $. If $ D\subset U, U $ is open then there exists a closed set $ P $ containing $ D $ such that $ P-U\in I $. Since $ P $ is closed in $ Y $, $ P\cap C $ is closed in $ C $ and $ P\cap C\supset D $. Again $ U\cap C $ is open in $ C $ and $ U\cap C\supset D $. As $ (P-U)\in I  $ and $ (P-U)\cap C\subset P-U\in I, (P\cap C)- (U\cap C)=(P-U)\cap C\in I \Rightarrow D $ is $ Ig^*$-closed relative to $ C $ with respect to the ideal $ I_C $.
\end{proof}

\begin{theorem}\label{21}
Assume $ C $ is open and $ g^* $-closed, then $ C\cap D $ is $ Ig^* $-closed  if $ D $ is $ Ig^* $-closed.
\end{theorem}

\begin{proof}
Let $ C $ be open and $ g^* $-closed in $ (Y,\eta,I) $, then $ C $ is closed as proved in the first part of Theorem \ref{20}. So $ C\cap D $ is closed in $ D $, hence $ C\cap D $ is $ Ig^* $-closed relative to $ D $ with respect to the ideal $ I_D=\{O\subset D: O\in I\} $. So, by Theorem \ref{20}, $ C\cap D $ is $ Ig^* $-closed.  
\end{proof}

\section{\bf    $\textit Ig^
*$-open sets in a $ \sigma $-space}

We discuss in this section about $ Ig^
*$-open sets exclusively in ideal $ \sigma $-space.

\begin{theorem}\label{22}
A set $ C $ of $ (Y,\eta,I) $ is $ Ig^*$-open if and only if there is an open set $ V $ contained in $ C $ such that $ P-U \subset V $ for some $ U\in I $ whenever $  P\subset C $ and $ P $ is closed.
\end{theorem}

\begin{proof}
Assume $ C $ is an $ Ig^*$-open set of $ (Y,\eta,I) $  and $ P\subset C, P $ is closed. So $ Y-C\subset Y-P $, an open set and $ Y-C $ is $ Ig^*$-closed. Therefore, there is a closed set $ F $ containing $ Y-C $ i.e. $ Y-F\subset C $ such that $ F-(Y-P)\in I $ i.e. $ F\subset (Y-P) \bigcup U $ for some $ U\in I $. This implies that $ Y-((Y-P)\bigcup U)\subset Y-F \Rightarrow P\bigcap (Y-U)\subset Y-F \Rightarrow P-U\subset Y-F=V $ (say), an open set.

Conversely, suppose the conditions hold. We will prove $ Y-C $ is $ Ig^*$-closed. Let $ G $ be an open set  such that $ Y-C\subset G $, then $ C\supset Y-G $, a closed set. Then by assumption, there is an open set $ V, V\subset C $, such that $ (Y-G) -U\subset V $ for some $ U\in I $. So $ (Y-G)\bigcap (Y-U)\subset V\Rightarrow Y-(G\bigcup U)\subset V\Rightarrow Y-V\subset G\bigcup U $. So $ (Y-V)-G\subset U, Y-V $ is a closed set containing $ Y-C $. Thus $ Y-C $ is $ Ig^*$-closed and hence $ C $ is $ Ig^*$-open.
\end{proof}

Union of two $ Ig^* $-open sets may not be always $ Ig^* $-open which is seen from the Example \ref{18}. However the next theorem shows that this will be true under certain additional condition.

\begin{theorem}\label{23}
Suppose $ C_1,C_2 $ are two weakly separated $ Ig^* $-open sets of $ (Y,\eta,I) $, then $ C_1\cup C_2 $ is $ Ig^* $-open. 
\end{theorem}
\begin{proof}
Suppose $ C_1,C_2 $ are two weakly separated $Ig^* $-open sets of $ (Y,\eta,I) $. So there are open sets $ U_1, U_2 $ such that  $ C_1\subset U_1,  C_2\subset U_2, U_1\cap C_2=U_2\cap C_1=\emptyset $. Let $ P_i=Y-U_i $ for $ i=1,2 $. Then $ P_i $ are closed sets and $ C_1\subset P_2, C_2\subset P_1 $. Again since $ C_1 $ and $ C_2 $ are $Ig^* $-open sets then there are open sets $ V_1, V_2 $ contained in $ C_1, C_2 $ respectively such that $ F_1-G_1\subset V_1 $ and $ F_2-G_2\subset V_2 $ for some $ G_1, G_2\in I $ whenever $ F_1, F_2 $ are closed sets and $ F_1\subset C_1 $ and $ F_2\subset C_2 $. Clearly $ V_1\cup V_2 $ is open and $ V_1\cup V_2\subset C_1\cup C_2 $. Let $ P\subset C_1\cup C_2 $ and $ P $ be closed. Now $ P=P\cap (C_1\cup C_2)=(P\cap C_1)\cup (P\cap C_2)\subset (P\cap P_2)\cup (P\cap P_1) $ where $ P\cap P_i $ is closed for each $ i=1,2 $. Further $ P\cap P_1\subset (C_1\cup C_2)\cap P_1\subset C_2 $. Similarly, $ P\cap P_2\subset C_1 $. Since $ C_1, C_2 $ are $ Ig^* $-open sets, $ P\cap P_1\subset V_2\bigcup U_2 $ and  $ P\cap P_2\subset V_1\bigcup U_1 $, for some $ U_1, U_2\in I $. Now $ P\subset (P\cap P_2)\cup (P\cap P_1)\subset (V_1\cup V_2)\bigcup(U_1\cup U_2)$ and so  $(C_1\cup C_2)$ is $ Ig^* $-open.  
\end{proof}

\begin{theorem}\label{24}
Let $ (Y,\eta,I) $ be an ideal $ \sigma $-space with $ I=\{\emptyset\} $. A set $ C\subset Y $ is $ Ig^* $-open if and only if there is an open set $ V\subset C $ such that $ V\cup (Y-C)\subset U, U $ is open implies  $ U=Y $.
\end{theorem}

\begin{proof}
Let $ C $ be $ Ig^* $-open in $ (Y,\eta,I) $ where $ I=\{\emptyset\} $. Then there is an open set $ V, V\subset C $, satisfying the condition of Theorem \ref{22}. Assume $ U $ is an open set such that $ V\cup (Y-C)\subset U $. Taking complement, $  C\cap (Y-V)\supset (Y-U) $. So $ (Y-V)\supset Y-U $ and $ C\supset (Y-U) $, a closed set. Since $ C $ is $ Ig^* $-open, $ (Y-U)-\{\emptyset\}\subset V $. Hence $ Y-U\subset V\cap (Y-V)=\emptyset \Rightarrow U=Y $.

Conversely, let there be an open set $ V\subset C $ such that $ V\cup (Y-C) \subset U, U $ is an open set implies $ U=Y $. Let $ P $ be a closed set contained in $ C $. Now $ V\cup (Y-C)\subset V\cup (Y-P), $ an open set. Then by the given conditions, $ V\cup (Y-P)=Y $ which implies $ P\subset V\Rightarrow P-V\subset \{\emptyset\} $ i.e. $ P-\{\emptyset\}\subset V $. Hence $ C $ is $ Ig^* $-open since $ I=\{\emptyset\} $.
\end{proof}

\begin{theorem}\label{25}
Let $ C $ be a subset in an ideal $ \sigma $-space $ (Y,\eta,I) $ and $ \overline{C} $ be closed, then $ C $ is $ Ig^* $-closed if and only if $ \overline{C}-C $ is $ Ig^* $-open. 
\end{theorem}

\begin{proof}
Assume $ \overline{C} $ is closed, $ C $ is $ Ig^* $-closed and $ P\subset \overline{C}-C,  P $ is a closed set. Then by Theorem \ref{14}, $ P\in I\Rightarrow P-U=\emptyset $ for some $ U\in I $. Since $ \{\emptyset\} $ is an open set of $(\overline{C}-C) $, by Theorem \ref{22}, $ \overline{C}-C  $ is $ Ig^* $-open.

Conversely, let $ \overline{C}-C  $ be $ Ig^* $-open and $ \overline{C} $ be closed. We will prove that $ C $ is $ Ig^* $-closed. Let $ C\subset G, G $ is open. Now $ \overline{C}\cap (Y-G) \subset \overline{C}\cap (Y-C)=\overline{C}-C  $ and $ \overline{C}\cap (Y-G) $ is a closed set contained in  $ \overline{C}-C  $. Since $ \overline{C}-C  $ is $ Ig^* $-open, by Theorem \ref{22}, there is an open set $ V=\{\emptyset\} $ contained in $ \overline{C}-C  $ such that $ \overline{C}\cap (Y-G)-U\subset V =\{\emptyset\} $ for some $ U\in I \Rightarrow \overline{C}\cap (Y-G)\subset U\in I  $ and hence $ \overline{C}-G\in I $. Thus $ C $ is $ Ig^* $-closed.
\end{proof}

\begin{theorem}\label{26}
In $ (Y,\eta,I) $, if $ Int(C)\subset D\subset C $ and $ C $ is $ Ig^* $-open, then $ D $ is $ Ig^* $-open.
\end{theorem}
\begin{proof}
Let $ Int(C)\subset D\subset C $ and  $ C $ be $ Ig^* $-open set,  then $ Y-C\subset Y-D\subset \overline{Y-C} $ and $ Y-C $ is $ Ig^* $-closed. By Theorem \ref{19}, $ Y-D $ is $ Ig^* $-closed and hence $ D $ is $ Ig^* $-open.
\end{proof}

\begin{theorem}\label{27}  Assume $ D\subset C\subset Y $ and $ C $ is closed and $ Ig^* $-open, then 
$ D $ is $ Ig^* $-open relative to $ C $ with respect to the ideal $ I_C=\{O\subset C: O\in I\} $ if and only if $ D $ is $ Ig^* $-open.
\end{theorem}
\begin{proof}
Suppose  $ (Y,\eta,I) $ is an ideal $ \sigma $-space and $ D $ is $ Ig^* $-open relative to $ C $ and $ P\subset D\subset C, P $ is closed in $ Y $. So $ P $ is closed in $ C $. Then there is an open set $ V_1 $ relative to $ C $ contained in $ D $ such that $ P-U_1\subset V_1 $ for some $ U_1\in I_C $. Now $ V_1=G_1\cap C $ for some open set $ G_1 $ in $ Y $, so $ P-U_1\subset G_1\cap C\subset D $. Since $ C  $ is $ Ig^* $-open and $ P\subset D\subset C $, there exists an open set $ V_2 $ in $ Y $ contained in $ C $ such that  $ P-U_2\subset V_2\subset C $ for some $ U_2\in I $ . Now $ P-(U_1\bigcup U_2)\subset (P-U_1)\bigcap (P-U_2)\subset (G_1\bigcap C)\bigcap V_2=(G_1\bigcap V_2)\bigcap C\subset (G_1\bigcap V_2)\subset G_1\bigcap C\subset D $. Therefore, $ G_1\bigcap V_2 $ is an open set of $ Y $ contained in $ D $ and $ U_1\bigcup U_2\in I $. So $ D $ is $ Ig^* $-open.

Conversely, suppose $ C $ and $ D $ are $ Ig^* $-open, $ C $ is closed, $ D\subset C $. Let  $ F $ be a closed  set relative to $ C, F\subset D $. So $ F $ is  closed in $ Y $. Since $ D $ is $ Ig^* $-open, there is an open set $ V_1 $ contained in $ D $ such that $ F-U_1\subset V_1 $ for some $ U_1\in I $. Again since $ C $ is $ Ig^* $-open and $ F\subset D\subset C $, there is an open set $ V_2 $ contained in $ C $ such that $ F-U_2\subset V_2= V_2\cap C $ for some $ U_2\in I $. Now $ F-(U_1\cup U_2)=(F-U_1)\cap (F-U_2)\subset V_1\cap V_2\cap C\subset V_1\cap C $, an open set in $ C $ and $ V_1\cap C\subset V_1\subset D $ and $ (U_1\cup U_2)\in I_C $. Hence the result follows.
\end{proof}

\section{\bf $ Ig^*$-$T_0, Ig^*$-$T_1$ and $ Ig^*$-$T_\omega $ axioms}

We introduce here some separation axioms viz. $ Ig^*$-$T_0, Ig^*$-$T_1$ and $ Ig^*$-$T_\omega $ axioms in an ideal $ \sigma $-space and investigate some of their related topological properties and try to establish a relation among them.

\begin{definition}\label{28}  An ideal $ \sigma $-space $ (Y,\eta,I) $ is said to be

(I):  $ Ig^*$-$T_0 $ if for any pair of distinct points $ p,q\in Y $, there exists an $ Ig^* $-open set $ U $ which contains one of them but not the other;

(II):  $ Ig^*$-$T_1 $ if for any pair of distinct points $ p,q\in Y $, there exist $ Ig^* $-open sets $ U $ and $ V $  such that $ p\in U, q\not\in U, q\in V, p\not\in V $.
\end{definition}

\begin{definition}\label{29}
A point $ p $ in an ideal $ \sigma $-space $ (Y,\eta,I) $  is said to be an $ Ig^* $-limit point of a set $ D\subset Y $ if every $ Ig^* $-open set $ U $ containing $ p $ such that $ U\cap (D-\{p\})\not=\emptyset $. The set of all $ Ig^* $-limit points of a set $ D\subset Y $ is denoted by $ D'_{Ig^*} $.
\end{definition}

\begin{definition}\label{30}
For any set $ C $ of an ideal $ \sigma $-space $ (Y,\eta,I) $ we define $ Ig^* $-closure of $ C $  denoted by $ \overline{C^{Ig^*}}=\bigcap \{P: C\subset P $ and $ P $ is $ Ig^* $-closed\}.  
\end{definition}

It reveals from the Example \ref{31} that $ Ig^* $-closure of a set may not be $ Ig^* $-closed.

\begin{example}\label{31}
Example of $ Ig^* $-closure of a set which is not $ Ig^* $-closed in an ideal $ \sigma $-space. 

We consider the Examole \ref{10} (i) where we can easily verify that the ideal $ \sigma $-space is $ Ig^*$-$T_0 $. Assume $ p,q\in Y, p\not=q $ and for $ Ig^*$-$T_0 $ ideal $ \sigma $-space there is an $ Ig^*$-open set $ V =\{p\} $ such that $ p\in V $ and $ q\not\in V $, so $ V\cap \{q\}=\emptyset $ and hence $ p $ is not an $ Ig^*$-limit point of $ \{q\} $. This is true for any point $ p, p\not=q $. Therefore $ \{q\}'_{Ig^*}=\emptyset\subset \{q\}\Rightarrow  \overline{\{q\}^{Ig^*}}=\{q\} $ but $ \{q\} $ is not $ Ig^*$-closed.
\end{example}

\begin{lemma}\label{32}
 Let   $ A $ be a subset of an ideal $ \sigma $-space $ (Y,\eta,I) $, then $ \overline{( \overline{A^{Ig^*}})^{I_g^*}}=\overline{A^{Ig^*}}  $.
 \end{lemma}
 \begin{proof}
 Since $ A \subset \overline{A^{Ig^*}} $ then $ \overline{A^{Ig^*}}\subset \overline{( \overline{A^{Ig^*}})^{I_g^*}} $. Again, by definition, $ \overline{( \overline{A^{Ig^*}})^{I_g^*}}= \bigcap_{i\in\Lambda}\{ V_i; V_i $'s are $ Ig^* $-closed  and $ V_i \supset \overline{A^{Ig^*}}\} $ and $ \overline{A^{Ig^*}} $ = $\bigcap _{j\in\Lambda'}\{ U_j; U_j $'s are $ Ig^* $-closed  and $ U_j \supset A \} $ where $ \Lambda $ and $ \Lambda' $ are indexing sets. The sets $ U_j, j \in \Lambda' $ also contain $ \overline{A^{Ig^*}} $ i.e. $ U_j\supset \overline{A^{Ig^*}} , j \in \Lambda'$. Therefore the collection $ \{U_j, j\in\Lambda'\} \subset \{V_i, i\in\Lambda\} $ and so $ \bigcap _{j\in\Lambda'} U_j \supset\bigcap_{i\in\Lambda} V_i $ implies $ \overline{A^{Ig^*}} \supset \overline{( \overline{A^{Ig^*}})^{I_g^*}} $. Hence  $ \overline{( \overline{A^{Ig^*}})^{I_g^*}}=\overline{A^{Ig^*}}  $.
 \end{proof}

\begin{theorem}\label{33}
An ideal $ \sigma $-space $ (Y,\eta,I) $ is $ Ig^*$-$T_0 $ if and only if $ Ig^* $-closures of any two disjoint singletons are disjoint. 
\end{theorem}
\begin{proof}
Let $ (Y,\eta,I) $ be an $ Ig^*$-$T_0 $ ideal $ \sigma $-space and $ p,q\in Y; p\not= q $. Then there is an $ Ig^* $-open set $ U $ such that $ p\in U $ but $ q\not\in U $. So $ q\in Y-U $ and $ p\not\in Y-U $, an $ Ig^* $-closed set. Hence $ p\not\in \overline{\{q\}^{Ig^*}} $ and $ q\in \overline{\{q\}^{Ig^*}} $. This implies that $ \overline{\{p\}^{Ig^*}} \not= \overline{\{q\}^{Ig^*}} $.

Conversely, let $ p,q\in Y; p\not= q $ and $ \overline{\{p\}^{Ig^*}} \not= \overline{\{q\}^{Ig^*}} $. Then there is a point $ y\in Y $ such that $ y\in \overline{\{p\}^{Ig^*}} $ but $ y\not\in \overline{\{q\}^{Ig^*}} $. Now $ p\in  \overline{\{p\}^{Ig^*}} $, we claim that $ p\not\in \overline{\{q\}^{Ig^*}} $. If possible, let $ p\in \overline{\{q\}^{Ig^*}} \Rightarrow \overline{\{p\}^{Ig^*}}\subset \overline{( \overline{\{q\}^{Ig^*}})^{I_g^*}}=\overline{\{q\}^{Ig^*}} $, a contradiction to the assumption. Hence $ p\not\in \overline{\{q\}^{Ig^*}} $. So there is an $ Ig^* $-closed set $ F $ containing $ \{q\} $ such that $ p\not\in F $, then $ p\in Y-F $, an $ Ig^* $-open set and $ q\not\in Y-F $. Hence $ (Y,\eta,I) $ is an $ Ig^*$-$T_0 $ ideal $ \sigma $-space. 
\end{proof}

\begin{theorem}\label{34}
Every $ Ig^* $-closed set $ P $ of $ (Y,\eta,I) $ contains all its $ Ig^* $-limit points.
\end{theorem}

\begin{proof}
Let $ P $ be an $ Ig^* $-closed set of an ideal $ \sigma $-space $ (Y,\eta,I) $ and $ p\in Y-P $, an $ Ig^* $-open set. So $ (Y-P) $ does not intersect $ P $, then $ p $ can not be an $ Ig^* $-limit point of $ P $ and $ p $ is any point outside $ P $. Hence the result follows.
\end{proof}

But the converse may not be true in general as revealed from the Example \ref{35}.

\begin{example}\label{35}
Example of a set $ D\supset D'_{Ig^*} $ but $ D $ is not $ Ig^* $-closed  in $ (Y,\eta,I) $.

Let $ Y=R $, the set of real numbers,  $ E_i $ are all countable subsets of irrational numbers, $   \eta=\{\emptyset, Y, E_i\} $ and $ I=\{\emptyset,\{\sqrt{2}\}\} $. Then $ (Y,\eta,I) $ is an ideal $ \sigma $-space. Consider a  set $ D=\{\sqrt{5}, \sqrt{7}, \sqrt{13}, \sqrt{17}, \sqrt{19}\} $, an open set but not a closed set. Therefore $ D $ is not $ Ig^* $-closed since $ I=\{\emptyset, \{\sqrt{2}\}\} $. Here a point $ p, p\not\in D $ can not be an $ Ig^* $-limit point of $ D $ implies $ D\supset D'_{Ig^*} $. Hence the result follows. 
\end{example}

\begin{lemma}\label{36}
Let $ D $ be a subset of an ideal $ \sigma $-space $ (Y,\eta,I) $, a point $ p\in \overline{D^{Ig^*}} $ if and only if every $ Ig^* $-open set containing $ p $ intersects $ D $.
\end{lemma}

\begin{proof}
Assume $ p\in \overline{D^{Ig^*}} $. If $ p\in D $, then the case is trivial. So let $ p\in \overline{D^{Ig^*}}-D $ and let $ V $ be an $ Ig^* $-open set containing $ p $. If $ V $ does not intersect $ D $, then $ D\subset Y-V $, an $ Ig^* $-closed set and then $ \overline{D^{Ig^*}} $ does not contain the point $ p $, since $  \overline{D^{Ig^*}}\subset Y-V $ which is a contradiction that $ p\in \overline{D^{Ig^*}} $. Hence every $ Ig^* $-open set containing $ p $ intersects $ D $. 

Conversely, suppose the conditions hold and $ p\not\in \overline{D^{Ig^*}} $. Then  there exists an $ Ig^* $-closed set $ F $ containing $ D $ such that $ p\not\in F $. So $ p\in Y-F $, an $ Ig^* $-open set which does not intersect $ D $, a contradiction to the assumption. Hence  $ p\in \overline{D^{Ig^*}} $. 
\end{proof}

\begin{theorem}\label{37}
Let $ D $ be a subset of an ideal $ \sigma $-space $ (Y,\eta,I) $, then $ D'_{Ig^*}\cup D=\overline{D^{Ig^*}} $. 
\end{theorem}

\begin{proof}
Let $ p\in  D'_{Ig^*}\cup D $. If $ p\in D $, then $ p\in \overline{D^{Ig^*}} $. So let $ p\in D'_{Ig^*} $ and $ p\not\in \overline{D^{Ig^*}} $. Then there exists an $ Ig^* $-closed set $ F $ containing $ D $ such that $ p\not\in F $.  So $ p\in Y-F $, an $ Ig^* $-open set which does not intersect $ D \Rightarrow   p\not\in  D'_{Ig^*} $, a contradiction that $ p $ is an $ Ig^* $-limit point of $ D $. Hence $ p\in \overline{D^{Ig^*}}\Rightarrow D'_{Ig^*}\cup D\subset \overline{D^{Ig^*}} $. Conversely, let $ p\in \overline{D^{Ig^*}} $. Then by Lemma \ref{36}, either $ p\in D $ or $ p\in D'_{Ig^*}\Rightarrow\overline{D^{Ig^*}}\subset D'_{Ig^*}\cup D $. Hence $ \overline{D^{Ig^*}}=D'_{Ig^*}\cup D $.  
\end{proof}

\begin{theorem}\label{38}
Every subspace of $ Ig^*$-$T_1 $ ideal $ \sigma $-space $ (Y,\eta,I) $ is $ Ig^*$-$T_1 $. 
\end{theorem}
\begin{proof}
Let $ (Y,\eta,I) $ be an $ Ig^*$-$T_1 $ ideal $ \sigma $-space, $ M $ is a subspace of $ Y $ and $ p,q\in M, p\not=q $.  Then there are $ Ig^* $-open sets $ U $ and $ V $ such that $ p\in U, q\not\in U, q\in V, p\not\in V $. Now $ M\cap U $ and $ M\cap V $ are also $ Ig^* $-open sets in $ M $ (due to subspace) and $ p\in M\cap U, q\not\in M\cap U, q\in M\cap V, p\not\in M\cap V $, thus the subspace $ M $ is $ Ig^*$-$T_1 $.
\end{proof}

\begin{note}\label{39}
It can be proved in a similar way as Theorem \ref{38} that every subspace of $ Ig^*$-$T_0 $ ideal $ \sigma $-space $ (Y,\eta,I) $ is $ Ig^*$-$T_0 $. We can easily verify that $ T_0 $ and  $ T_1 $ $ \sigma $-spaces are respectively $ Ig^*$-$T_0 $ and $ Ig^*$-$T_1 $ ideal $ \sigma $-spaces and each $ Ig^*$-$T_1 $ ideal $ \sigma $-space is $ Ig^*$-$T_0 $. 
But converses may not be always true as seen from the Example \ref{40} (i), \ref{40} (ii) and \ref{40} (iii). 
\end{note}

\begin{example}\label{40} (i):
Example of an $ Ig^*$-$T_0 $ ideal $ \sigma $-space which is not $ T_0 $ $ \sigma $-space.

Suppose $ Y=R-Q,  \eta=\{\emptyset, Y, E_i\cup\{\sqrt{2},\sqrt{3}\}\}, E_i $ are all countable subsets of $ Y $ and $ I=\{\emptyset, \{\sqrt{7}\}\} $. Then $ (Y,\eta,I) $ is an ideal $ \sigma $-space. Now by Theorem \ref{22}, each singleton is an $ Ig^* $-open set since $ \{\emptyset\} $ is an open set as well as a closed set. Hence $ (Y,\eta,I) $ is an $ Ig^*$-$T_0 $ ideal $ \sigma $-space but not a $ T_0 $ $ \sigma $-space which can be verified by considering a pair of distinct points $ \sqrt{2},\sqrt{3} $.

(ii): Example of an $ Ig^*$-$T_1 $ ideal $ \sigma $-space which is not $ T_1 $ $ \sigma $-space.

Consider the Example \ref{18} where by Theorem \ref{22}, each singleton is an $ Ig^* $-open set since $ \{\emptyset\} $ is an open set as well as a closed set. Hence $ (Y,\eta,I) $ is an $ Ig^*$-$T_1 $ ideal $ \sigma $-space but not a $ T_1 $ $ \sigma $-space which may be checked by  considering a pair of distinct points $ \sqrt{2},\sqrt{11} $. 

(iii):  Example of an $ Ig^*$-$T_0 $ ideal $ \sigma $-space which is not $ Ig^*$-$T_1 $.

Suppose $ Y=\{a,b,c\}, \eta= \{\emptyset, Y, \{b,c\}\} $ and $ I=\{\emptyset,b\} $. Then $ (Y,\eta,I) $ is  an ideal $ \sigma $-space. Here, a family of $ Ig^* $-open sets is $ \{\emptyset,Y,\{b\}, \{c\},\{b.c\}\} $. Clearly, $ (Y,\eta,I) $ is an $ Ig^*$-$T_0 $ ideal $ \sigma $-space but not an $ Ig^*$-$T_1 $. For, consider a pair of distinct points $ a,b $. Then the only $ Ig^* $-open set containing $ a $ is the set $ Y $ which contains the point $ b $ also.   
\end{example}

\begin{theorem}\label{41}
If each singleton of an ideal $ \sigma $-space $ (Y,\eta,I) $ is $ Ig^* $-closed, then it  is $ Ig^*$-$T_1 $. 

Proof is straight forward, so is omitted.
\end{theorem}
But it is not reversible as it can be seen from the Example \ref{10} (i) .

\begin{definition}\label{42}
An ideal $ \sigma $-space is said to be $ Ig^*$-$T_\omega $ if every $ Ig^*$-closed  set is closed.
\end{definition}

\begin{theorem}\label{43}   An
$ Ig^*$-$T_\omega $ ideal $ \sigma $-space $ (Y,\eta,I) $ is $ Ig^*$-$T_0 $.
\end{theorem}
\begin{proof}
Assume $ (Y,\eta,I) $ is an $ Ig^*$-$T_\omega $ ideal $ \sigma $-space and it is not  $ Ig^*$-$T_0 $. Then by Theorem \ref{32}, there is one pair of points $ p,q; p\not=q $ such that $ \overline{\{p\}^{Ig^*}}= \overline{\{q\}^{Ig^*}} $. Now $ \{p\}\cap \{q\}=\emptyset\Rightarrow \{q\}\subset Y-\{p\} $. We claim that $ \{p\} $ is not closed. For, if $ \{p\} $ is closed then it is $ Ig^* $-closed, so $ \overline{\{p\}^{Ig^*}}=\{p\} $ and hence $ \overline{\{q\}^{Ig^*}}=\{p\} $ which is not possible. Therefore by Theorem \ref{16}, $ Y-\{p\} $ is $ Ig^* $-closed and so is closed, since $ (T,\eta,I) $ is $ Ig^*$-$T_\omega $ space,  which is not also possible, since $ \{q\}\subset Y-\{p\}\Rightarrow  \overline{\{q\}^{Ig^*}}\subset \overline{(Y-\{p\})^{Ig^*}}=Y-\{p\}\Rightarrow  \overline{\{q\}^{Ig^*}}\not= \overline{\{p\}^{Ig^*}} $, a contradiction to the assumption. Hence $ \overline{\{p\}^{Ig^*}}\not= \overline{\{q\}^{Ig^*}} $ and so the ideal $ \sigma $-space is $ Ig^*$-$T_0 $.    
\end{proof}

\begin{remark}\label{44}
But converse may not be true as it can be verified from Example \ref{10} (i) which shows that the ideal $ \sigma $-space $ (Y,\eta, I) $ is $ Ig^*$-$T_1 $ so it is  $ Ig^*$-$T_0  $ but it is not $ Ig^*$-$T_\omega $.
\end{remark}

\begin{example}\label{45} Example of an
$ Ig^*$-$T_\omega $ ideal $ \sigma $-space which is not $ Ig^*$-$T_1 $.

Let $ Y=\{p,q\}, \eta=\{\emptyset, Y,\{q\}\} $ and $ I=\{\emptyset\} $. Then $ (Y,\eta, I) $ is an ideal $ \sigma $-space. Here the $ Ig^* $-closed sets are $ \{\{\emptyset\}, \{p\}, Y\} $ and these sets are also closed sets, so $ (Y,\eta, I) $ is an $ Ig^*$-$T_\omega $ ideal $ \sigma $-space but it is not $ Ig^*$-$T_1 $ since $ Ig^* $-open sets are $ \emptyset, Y $ and $\{q\} $.  
\end{example}

Now we give a chacterization for the trivial ideal $ \sigma $-space to be an $ Ig^*$-$T_\omega $.

\begin{theorem}\label{46}
:  An ideal $ \sigma $-space $ (Y, \eta,I) $ with $ I=\{\emptyset\} $ is $ Ig^*$-$T_\omega $ if and only if 

(i)  for each $ p \in Y $,  $ \{p\} $ is either open or closed 

(ii)  $ K =K_I^* $

where $ K =\{D: \overline{(Y-D)}$ is closed\} and $ K_I^*=\{D:\overline{(Y-D)^{Ig^*}}$ is $ Ig^*$-closed\}.
\end{theorem}
\begin{proof}
Suppose $ (Y, \eta,I) $ is an $ Ig^*$-$T_\omega $ ideal $ \sigma $-space. 

(i) Let $ p\in Y $ and $ \{p\} $ be not closed. Then by Theorem \ref{16}, $ Y-\{p\} $ is $ Ig^* $-closed and for $ Ig^*$-$T_\omega $ ideal $ \sigma $-space, $ Y-\{p\} $ is closed and hence $ \{p\} $ is open.

(ii) Let $ D\in K\Rightarrow \overline{(Y-D)} $ is closed $ \Rightarrow \overline{(Y-D)} $ is $ Ig^* $-closed $ \Rightarrow  \overline{(Y-D)^{Ig^*}}=\overline{(Y-D)} $ is $ Ig^* $-closed $ \Rightarrow D\in K_I^*\Rightarrow K \subset K_I^*  $. On the other hand, suppose $ D \subset K_I^* \Rightarrow \overline{(Y-D)^{Ig^*}} $ is $ Ig^* $-closed $ \Rightarrow \overline{(Y-D)} $ is closed $ \Rightarrow D\in K \Rightarrow  K_I^*\subset K $. Hence $ K= K_I^* $.

Conversely, assume the conditions (i) and (ii) hold and $ D $ be any $ Ig^* $-closed set. So $ \overline{D^{Ig^*}}=D $ is $ Ig^* $-closed, then $ (Y-D)\in K_I^*=K\Rightarrow \overline{D} $ is closed. We claim that $ \overline{D}=D $. If not, then there exists $ p\in \overline{D}-D $. Now $ \{p\} $ is not closed, if $ \{p\} $ is closed then by Theorem \ref{14}, $ \{p\}\in I=\{\emptyset\} $ which contradicts that $ \{p\} $ is non-empty. Then by (i), $ \{p\} $ is open $ \Rightarrow Y-\{p\} $ is closed. But $ \{p\}\not\in D\Rightarrow D\subset (Y-\{p\})\Rightarrow \overline{D}\subset \overline{(Y-\{p\})}=Y-\{p\} $, a contradiction. Thus $ \overline{D}=D $. Hence $ D $ is closed  and the result follows. 
\end{proof}

\begin{corollary}\label{47}
An ideal $ \sigma $-space $ (Y, \eta,I) $ with $ I=\{\emptyset\} $ is $ Ig^*$-$T_\omega $ if and only if 

(i)  every subset of $ Y $ is the intersection of all open sets and all  closed sets containing it and

(ii)  $ K =K_I^* $
where $ K $ and $ K_I^* $ are defined as Theorem \ref{46}.
\end{corollary}

\begin{theorem}\label{48}
If $ (Y,\eta, I) $ be an $ Ig^*$-$T_0 $ ideal $ \sigma $-space, then for every pair of distinct points $ p,q\in Y $, either $ p\not\in \overline{\{q\}^{Ig^*}} $ or $ q\not\in \overline{\{p\}^{Ig^*}} $.
\end{theorem}

\begin{proof}
Let $ (Y,\eta, I) $ be an $ Ig^*$-$T_0 $ ideal $ \sigma $-space and  $ p,q\in Y $ be a pair of distinct points. So either there exists an $ Ig^* $-open set $ U $ such that $ p\in U, q\not\in U $ or,  there exists an $ Ig^* $-open set $ V $ such that $ q\in V, p\not\in V $. Since $ p\not=q $ either $ q\not\in \overline{\{p\}^{Ig^*}} $ or, $  p\not\in \overline{\{q\}^{Ig^*}} $. 
\end{proof}

Now we introduce the next definition for an $ Ig^*$-$T_0 $ ideal $ \sigma $-space to be $ Ig^*$-$T_1 $.  

\begin{definition}\label{49}
An ideal $ \sigma $-space $ (Y,\eta,I) $ is said to be an $ Ig^* $-symmetric if for $ p,q\in Y,  p\in \overline{\{q\}^{Ig^*}}\Rightarrow q\in \overline{\{p\}^{Ig^*}} $.
\end{definition}

\begin{theorem}\label{50}
An $ Ig^* $-symmetric and $ Ig^*$-$T_0 $ ideal $ \sigma $-space $ (Y,\eta,I) $ is $ Ig^*$-$T_1 $.
\end{theorem}

\begin{proof}
Let $ (Y,\eta,I) $ be an  $ Ig^* $-symmetric and $ Ig^*$-$T_0 $ ideal $ \sigma $-space and let $ p,q\in Y; p\not=q $. Since $ (Y,\eta,I) $ is $ Ig^*$-$T_0 $ ideal $ \sigma $-space, by Theorem \ref{48}, we get either $ p\not\in \overline{\{q\}^{Ig^*}} $ or $ q\not\in \overline{\{p\}^{Ig^*}} $. Suppose $ p\not\in \overline{\{q\}^{Ig^*}} $. Then we claim that $ q\not\in \overline{\{p\}^{Ig^*}} $. For if $ q\in \overline{\{p\}^{Ig^*}} $, then $ p\in \overline{\{q\}^{Ig^*}} $, since the ideal $ \sigma $-space is $ Ig^* $-symmetric. But this is a contradiction. Since $ p\not\in \overline{\{q\}^{Ig^*}} $, there is an $ Ig^* $-closed set $ F, F\supset \{q\} $ such that $ p\not\in F\Rightarrow p\in Y-F $, an $ Ig^* $-open set and $ q\not\in Y-F $. Similarly, if $ q\not\in \overline{\{p\}^{Ig^*}} $, then there is an $ Ig^* $-closed set $ P, P\supset \{p\} $ such that $ q\not\in P\Rightarrow q\in Y-P $, an $ Ig^* $-open set and $ p\not\in Y-P $. Hence the ideal $ \sigma $-space is  $ Ig^*$-$T_1 $.  
\end{proof}

\begin{remark}
Remark \ref{44} and Example \ref{45} reveal that  $ Ig^*$-$T_\omega $ and $ Ig^*$-$T_1 $ axioms are independent of each other. So as in \cite{NL}, $ Ig^*$-$T_\omega $ axiom can not be placed between $ Ig^*$-$T_0 $ and $ Ig^*$-$T_1 $ axioms.  As in \cite{DP}, $ Ig^* $-symmetric $ Ig^*$-$T_1 $ ideal $ \sigma $-space may not be $ Ig^*$-$T_\omega $ which can be seen from Example \ref{10} (i). This Example also shows that singleton is not $ Ig^* $-closed in the said ideal $ \sigma $-space. If strongly $ Ig^*$-symmetryness is introduced as in \cite{DP} in an ideal $ \sigma $-space, then strongly $ Ig^*$-symmetric $ Ig^*$-$T_1 $ ideal $ \sigma $-space also fails to be $ Ig^*$-$T_\omega $ since strongly $ Ig^*$-symmetric ideal $ \sigma $-space is $ Ig^* $-symmetric which is seen in next theorem. Thus an equivalence relation among $ Ig^*$-$T_0 $, $ Ig^*$-$T_\omega, Ig^*$-$T_1 $ axioms is not found. But a relation among the axioms is observed under certain additional conditions.   
\end{remark}

\begin{definition}\label{52}
An ideal $ \sigma $-space $ (Y,\eta,I) $ is said to be strongly $ Ig^* $-symmetric if for each $ y\in Y,  \{y\} $ is $ Ig^* $-closed. 
\end{definition}

\begin{theorem}\label{53}
A strongly $ Ig^* $-symmetric ideal $ \sigma $-space $ (Y,\eta,I) $ with $ I=\{\emptyset\} $ is $ Ig^*$-symmetric.
\end{theorem}
\begin{proof}
Suppose for any $ p,q\in Y, p\in \overline{\{q\}^{Ig^*}} $ but $ q\not\in \overline{\{p\}^{Ig^*}} $. Then there is an $ Ig^*$-closed set $ F $ containing $ \{p\} $ such that $ q\not\in F\Rightarrow q\in Y-F $, an $ Ig^* $-open set. Since the ideal $ \sigma $-space is strongly $ Ig^* $-symmetric, $ \{q\} $ is $ Ig^* $-closed and hence there is a closed set $ P $ containing $ \{q\} $ such that $ P-(Y-F)\in I=\{\emptyset\}\Rightarrow P\subset (Y-F) $. Since $ \{q\} $ is $ Ig^* $-closed, $ \{q\}=\overline{\{q\}^{Ig^*}}\subset P\Rightarrow p\in \overline{\{q\}^{Ig^*}}\subset P\subset Y-F $, a contradiction. Hence the result follows.
\end{proof}

\begin{corollary}\label{54}
A strongly $ Ig^*$-symmetric $ Ig^*$-$T_\omega $ ideal $ \sigma $-space is $ Ig^* $-$ T_1 $ if ideal is trivial.

Proof follows from Theorems \ref{43}, \ref{50} and \ref{53}.
\end{corollary}

\begin{remark}
In view of above discussion the following implications hold: $ Ig^*$-$T_\omega\Longrightarrow Ig^*$-$T_0\Longrightarrow Ig^*$-$T_1 $ in a strongly $ Ig^*$-symmetric ideal $ \sigma $-space $ (Y,\eta,I) $ with $ I=\{\emptyset\} $.  
\end{remark}


\begin{thebibliography}{99}\baselineskip=16pt
\bibitem{AD} A. D.Alexandroff, \textit{Additive set functions in abstract space, } Mat. Sb. (N.S.) \textbf{8}(50) (1940), 307-348 (English, Russian Summary).



\bibitem{AA} A. K. Banerjee and  A. Banerjee ,  \textit{ $ I $-convergence classes of sequences and nets in topological spaces,}  {Jordan Journal of Mathematics and Statistics,}   \textbf{11} (1) (2016), 13-31.

\bibitem{AAB} A. K. Banerjee and  A. Banerjee ,  \textit{A study on $ I $-Cauchy sequences and $ I $-divergence in $ S $-metric spaces,}  {Malaya Journal of Matematik (MJM),}   \textbf{6} (2) (2018), 326-330.

\bibitem{AJ1} A. K. Banerjee and J. Pal,  \textit{$\lambda^*$-closed sets and new separation axioms in Alexandraff spaces}, \textit{Korean J. Math.,}   \textbf{26} (4) (2018), 709-727. 

\bibitem{AP} A. K. Banerjee and A. Paul,  \textit{$ I $-divergence and $I^*$-divergence  in cone metric spaces}, \textit{Asian-European Journal of  Mathematics,}   \textbf{13} (08) 2050139. 

\bibitem{AC} A. Cs$\acute{a}$sz$ \acute{a} $r,  Generalized topology,   generalized continuity,  \textit{Acta Math. Hungar.,}   \textbf{96} (2002), 351-357. 

\bibitem{CA} A.Cs$\acute{a}$sz$ \acute{a} $r ,  Generalized open sets in generalized topologies,  \textit{Acta Math. Hungar.,}   \textbf{106} (12) (2005), 53-66.

\bibitem{PD} P.Das and M.A. Rashid, \textit{Certain separation axioms in a space,} Korean J. Math. Sciences, \textit{7} (2000), 81-93.

\bibitem{DP} P.Das and M.A. Rashid, \textit{$g^*$ closed sets and a new separation axioms in Alexandroff spaces,} Archivum Mathematicum (BRNO), Tomus 39 (2003), 299-307. 

\bibitem{WD} W. Dunham,  \textit{$ T_\frac{1}{2}-spaces $,} Kyungpook  Math. J., \textbf{17}(2) (1977), 161-169.


\bibitem{IV} T. Indira and S. Vijayalakshmi, \textit{$ I\ddot{g} $-separation axioms in ideal topological spaces,} Int. J. Math. and Appl, \textit{6}(1-B) (2018), 395-401.

\bibitem{JH} D.Jankovic and T.R. Hamlett, \textit{New topologies from old via ideal,} The American Mathematical Monthly, \textit{97}: 4, (2018), 295-310.


\bibitem{JR} S. Jafari and N. Rajesh, \textit{Generalized closed sets with respect to an ideal,} European journal of pure and applied mathmatics, \textit{4} (2) (2011), 147-151.

\bibitem{LD} B. K. Lahiri and P. Das, \textit{Semi-open set in a space}, Sains Malaysiana 24(4)  1995: 1-11.

\bibitem{BP} Lahiri  B. K. and P. Das, \textit{$ I $ and $ I^* $-convergence in topological spaces}, Mathemetica Bohemica, 130 (2)  2005: 153-160.



\bibitem{NL} N. Levine, N,  \textit{Generalised closed sets in topology}, Rend. Cire.  Mat. Palermo \textbf{19}(2) (1970), 89-96.

\bibitem{SM} M. S. Sarsak,  Weak separation axioms in generalized topological spaces, \textit{Acta Math. Hungar.,}131(1-2) (2011),110-121.

\bibitem{MS} M. S. Sarsak,  New separation axioms in generalized topological spaces, \textit{Acta Math. Hungar.,} \textbf{132(3)} (2011), 244-252.

\bibitem{DJ} B. C. Tripathy and D. J. Sarma, \textit {Generalized b-closed sets in ideal bitopological spaces}, \textit {Proyecciones J. Math.}, \textbf {33} (3) (2014), 315-324.

\bibitem{SV} S.Vaidyanathaswamy, set topology, Chelsea Publishing Company, (1960). 

  
\end{thebibliography}
\end{document}